\documentclass{article}

\usepackage{graphicx}
\usepackage[figuresright]{rotating}

\newtheorem{theorem}{Theorem}

\newenvironment{proof}[1][Proof.]{\begin{trivlist}
\item[\hskip \labelsep {\bfseries #1}]}{\end{trivlist}}

\usepackage{amsmath}
\usepackage{amsfonts}
\title{\LARGE \textbf{Interval Total Colorings of
Complete Multipartite Graphs and Hypercubes}}
\author{\normalsize Petros A. Petrosyan$^\dag$$^\ddag$, Nerses A. Khachatryan$^\ddag$}
\date{\small $^\dag$Institute for Informatics and Automation Problems of NAS RA,\\
$^\ddag$Department of Informatics and Applied Mathematics, YSU\\
e-mail: pet\_petros@ipia.sci.am, xachnerses@gmail.com}

\begin{document}
\textheight = 18.1cm \textwidth = 11.4cm \maketitle

\begin{abstract}
A total coloring of a graph $G$ is a coloring of its vertices and
edges such that no adjacent vertices, edges, and no incident
vertices and edges obtain the same color. An interval total
$t$-coloring of a graph $G$ is a total coloring of $G$ with colors
$1,\ldots,t$ such that all colors are used, and the edges incident
to each vertex $v$ together with $v$ are colored by $d_{G}(v)+1$
consecutive colors, where $d_{G}(v)$ is the degree of a vertex $v$
in $G$. In this paper we prove that all complete multipartite graphs
with the same number of vertices in each part are interval total
colorable. Moreover, we also give some bounds for the minimum and
the maximum span in interval total colorings of these graphs. Next,
we investigate interval total colorings of hypercubes $Q_{n}$. In
particular, we prove that $Q_{n}$ ($n\geq 3$) has an interval total
$t$-coloring if and only if $n+1\leq t\leq \frac{(n+1)(n+2)}{2}$.\\

\end{abstract}

\section{Introduction}\

All graphs considered in this paper are finite, undirected and have
no loops or multiple edges. Let $V(G)$ and $E(G)$ denote the sets of
vertices and edges of $G$, respectively. Let $VE(G)$ denote the set
$V(G)\cup E(G)$. The degree of a vertex $v$ in $G$ is denoted by
$d_{G}(v)$, the maximum degree of vertices in $G$ by $\Delta (G)$
and the total chromatic number of $G$ by $\chi^{\prime\prime}(G)$.
For $S\subseteq V(G)$, let $G[S]$ denote the subgraph of $G$ induced
by $S$, that is, $V(G[S])=S$ and $E(G[S])$ consists of those edges
of $E(G)$ for which both ends are in $S$. For $F\subseteq E(G)$, the
subgraph obtained by deleting the edges of $F$ from $G$ is denoted
by $G-F$. The terms and concepts that we do not define can be found
in \cite{ArsDenHag,West,Yap}.

Let $\left\lfloor a\right\rfloor $ denote the largest integer less
than or equal to $a$. For two positive integers $a$ and $b$ with
$a\leq b$, the set $\left\{a,\ldots,b\right\}$ is denoted by $[a,b]$
and called an interval. For an interval $[a,b]$ and a nonnegative
number $p$, the notation $[a,b]\oplus p$ means $[a+p,b+p]$.

A proper edge-coloring of a graph $G$ is a coloring of the edges of
$G$ such that no two adjacent edges receive the same color. If
$\alpha $ is a proper edge-coloring of $G$ and $v\in V(G)$, then
$S\left(v,\alpha \right)$ denotes the set of colors appearing on
edges incident to $v$. A proper edge-coloring of a graph $G$ is an
interval $t$-coloring \cite{ArsKam} if all colors are used, and for
any $v\in V(G)$, the set $S\left(v,\alpha \right)$ is an interval of
integers. A graph $G$ is interval colorable if it has an interval
$t$-coloring for some positive integer $t$. The set of all interval
colorable graphs is denoted by ${\cal N}$. For a graph $G\in {\cal
N}$, the least (the minimum span) and the greatest (the maximum
span) values of $t$ for which $G$ has an interval $t$-coloring are
denoted by $w(G)$ and $W(G)$, respectively. A total coloring of a
graph $G$ is a coloring of its vertices and edges such that no
adjacent vertices, edges, and no incident vertices and edges obtain
the same color. If $\alpha$ is a total coloring of a graph $G$, then
$S\left[v,\alpha \right]$ denotes the set $ S\left(v,\alpha
\right)\cup\{\alpha(v)\}$.

A graph $K_{n_{1},\ldots,n_{r}}$ is a complete $r$-partite ($r\geq
2$) graph if its vertices can be partitioned into $r$ independent
sets $V_{1},\ldots,V_{r}$ with $\vert V_{i}\vert=n_{i}$ ($1\leq
i\leq r$) such that each vertex in $V_{i}$ is adjacent to all the
other vertices in $V_{j}$ for $1\leq i<j\leq r$. A complete
$r$-partite graph $K_{n,\ldots,n}$ is a complete balanced
$r$-partite graph if $\vert V_{1}\vert =\vert V_{2}\vert =\cdots
=\vert V_{r}\vert=n$. Clearly, if $K_{n,\ldots,n}$ is a complete
balanced $r$-partite graph with $n$ vertices in each part, then
$\Delta(K_{n,\ldots,n})=(r-1)n$. Note that the complete graph
$K_{n}$ and the complete balanced bipartite graph $K_{n,n}$ are
special cases of the complete balanced $r$-partite graph. The total
chromatic numbers of complete and complete bipartite graphs were
determined by Behzad, Chartrand and Cooper \cite{BehCharCoop}. They
proved the following theorems:

\begin{theorem}
\label{mytheorem1} For any $n\in \mathbf{N}$, we have
\begin{center}
$\chi^{\prime\prime}(K_{n})=\left\{
\begin{tabular}{ll}
$n$, & if $n$ is odd, \\
$n+1$, & if $n$ is even. \\
\end{tabular}%
\right.$
\end{center}
\end{theorem}

\begin{theorem}
\label{mytheorem2} For any $m,n\in \mathbf{N}$, we have
\begin{center}
$\chi^{\prime\prime}(K_{m,n})=\left\{
\begin{tabular}{ll}
$\max \{m,n\}+1$, & if $m\neq n$, \\
$n+2$, & if $m=n$. \\
\end{tabular}%
\right.$
\end{center}
\end{theorem}

A more general result on total chromatic numbers of complete
balanced multipartite graphs was obtained by Bermond \cite{Bermond}.

\begin{theorem}
\label{mytheorem3} For any complete balanced $r$-partite graph
$K_{n,\ldots,n}$ ($r\geq 2,n\in \mathbf{N}$), we have
\begin{center}
$\chi^{\prime\prime}(K_{n,\ldots,n})=\left\{
\begin{tabular}{ll}
$(r-1)n+2$, & if $r=2$ or $r$ is even and $n$ is odd, \\
$(r-1)n+1$, & otherwise. \\
\end{tabular}%
\right.$
\end{center}
\end{theorem}

An interval total $t$-coloring of a graph $G$ is a total coloring
$\alpha$ of $G$ with colors $1,\ldots ,t$ such that all colors are
used, and for any $v\in V(G)$, the set $S\left[v,\alpha \right]$ is
an interval of integers. A graph $G$ is interval total colorable if
it has an interval total $t$-coloring for some positive integer $t$.
The set of all interval total colorable graphs is denoted by ${\cal
T}$. For a graph $G\in {\cal T}$, the least (the minimum span) and
the greatest (the maximum span) values of $t$ for which $G$ has an
interval total $t$-coloring are denoted by $w_{\tau }\left(G\right)$
and $W_{\tau }\left(G\right)$, respectively. Clearly,

\begin{center}
$\chi^{\prime \prime}\left(G\right)\leq w_{\tau }\left(G\right)\leq
W_{\tau }\left(G\right)\leq \vert V(G)\vert + \vert E(G)\vert$ for
every graph $G\in {\cal T}$.
\end{center}

The concept of interval total coloring was introduced by the first
author \cite{Petros1}. In \cite{Petros1,Petros2}, the author proved
that if $m+n+2-\gcd(m,n)\leq t\leq m+n+1$, then the complete
bipartite graph $K_{m,n}$ has an interval total $t$-coloring. Later,
the first author and Torosyan \cite{PetTor} obtained the following
results:

\begin{theorem}
\label{mytheorem4} For any $m,n\in \mathbf{N}$, we have
\begin{center}
$W_{\tau }(K_{m,n})=\left\{
\begin{tabular}{ll}
$m+n+1$, & if $m=n=1$,\\
$m+n+2$, & otherwise.\\
\end{tabular}%
\right.$
\end{center}
\end{theorem}

\begin{theorem}
\label{mytheorem5} For any $n\in \mathbf{N}$, we have
\begin{description}
\item[(1)] $K_{n,n}\in {\cal T}$,

\item[(2)] $w_{\tau }(K_{n,n})=\chi^{\prime\prime}(K_{n,n})=n+2$,

\item[(3)] $W_{\tau }(K_{n,n})=\left\{
\begin{tabular}{ll}
$2n+1$, & if $n=1$,\\
$2n+2$, & if $n\geq 2$,
\end{tabular}%
\right.$

\item[(4)] if $w_{\tau }(K_{n,n})\leq t\leq W_{\tau }(K_{n,n})$,
then $K_{n,n}$ has an interval total $t$-coloring.\
\end{description}
\end{theorem}

In \cite{Petros2}, the first author investigated interval total
colorings of complete graphs and hypercubes, where he proved the
following two theorems:

\begin{theorem}
\label{mytheorem6} For any $n\in \mathbf{N}$, we have
\begin{description}
\item[(1)] $K_{n}\in {\cal T}$,

\item[(2)] $w_{\tau }(K_{n})=\left\{
\begin{tabular}{ll}
$n$, & if $n$ is odd,\\
$\frac{3}{2}n$, & if $n$ is even,
\end{tabular}%
\right.$

\item[(3)] $W_{\tau }(K_{n})=2n-1$.
\end{description}
\end{theorem}

\begin{theorem}
\label{mytheorem7} For any $n\in \mathbf{N}$, we have
\begin{description}
\item[(1)] $Q_{n}\in {\cal T}$,

\item[(2)] $w_{\tau }(Q_{n})=\chi^{\prime\prime}(Q_{n})=\left\{
\begin{tabular}{ll}
$n+2$, & if $n\leq 2$,\\
$n+1$, & if $n\geq 3$,
\end{tabular}%
\right.$

\item[(3)] $W_{\tau }(Q_{n})\geq \frac{(n+1)(n+2)}{2}$,

\item[(4)] if $w_{\tau }(Q_{n})\leq t\leq \frac{(n+1)(n+2)}{2}$,
then $Q_{n}$ has an interval total $t$-coloring.\
\end{description}
\end{theorem}

Later, the first author and Torosyan \cite{PetTor} showed that if
$w_{\tau }(K_{n})\leq t\leq W_{\tau }(K_{n})$, then the complete
graph $K_{n}$ has an interval total $t$-coloring. In
\cite{PetShash1}, the first author and Shashikyan proved that trees
have an interval total coloring. In
\cite{PetShash2,PetShash3,PetShashTor}, they investigated interval
total colorings of bipartite graphs. In particular, they proved that
regular bipartite graphs, subcubic bipartite graphs, doubly convex
bipartite graphs, $\left(2,b \right)$-biregular bipartite graphs and
some classes of bipartite graphs with maximum degree $4$ have
interval total colorings. They also showed that there are bipartite
graphs that have no interval total coloring. The smallest known
bipartite graph with 26 vertices and maximum degree 18 that is not
interval total colorable was obtained by Shashikyan \cite{Shash}.

One of the less-investigated problems related to interval total
colorings is a problem of determining the exact values of the
minimum and the maximum span in interval total colorings of graphs.
The exact values of these parameters are known only for paths,
cycles, trees \cite{PetShash1,PetShash2}, wheels \cite{PetKhach2},
complete and complete balanced bipartite graphs
\cite{Petros1,Petros2,PetTor}. In some papers
\cite{PetKhach1,PetKhach3,PetKhach4,PetShashTorKhach} lower and
upper bounds are found for the minimum and the maximum span in
interval total colorings of certain graphs.

In this paper we prove that all complete balanced multipartite
graphs are interval total colorable and we derive some bounds for
the minimum and the maximum span in interval total colorings of
these graphs. Next, we investigate interval total colorings of
hypercubes $Q_{n}$ and we show that $W_{\tau
}(Q_{n})=\frac{(n+1)(n+2)}{2}$ for any $n\in \mathbf{N}$.
\bigskip

\section{Interval total colorings of complete multipartite graphs}\

We first consider interval total colorings of complete bipartite
graphs. It is known that $K_{m,n}\in {\cal T}$ and $w_{\tau
}(K_{m,n})\leq m+n+2-\gcd(m,n)$ for any $m,n\in \mathbf{N}$, but in
general the exact value of $w_{\tau }(K_{m,n})$ is unknown. Our
first result generalizes the point (2) of Theorem \ref{mytheorem5}.

\begin{theorem}
\label{mytheorem8} For any $n,l\in \mathbf{N}$, we have $w_{\tau
}(K_{n,n\cdot l})=\chi^{\prime\prime}(K_{n,n\cdot l})$.
\end{theorem}
\begin{proof} First of all let us consider the case $l=1$. By
Theorem \ref{mytheorem5}, we obtain $w_{\tau
}(K_{n,n})=\chi^{\prime\prime}(K_{n,n})=n+2$ for any $n\in
\mathbf{N}$.

Now we assume that $l\geq 2$.

Let $V\left(K_{n,n\cdot
l}\right)=\{u_{1},\ldots,u_{n},v_{1},\ldots,v_{n\cdot l}\}$ and
$E\left(K_{n,n\cdot l}\right)=\{u_{i}v_{j}|~1\leq i\leq n, 1\leq
j\leq n\cdot l\}$. Also, let $G=K_{n,n\cdot
l}\left[\{u_{1},\ldots,u_{n},v_{1},\ldots,v_{n}\}\right]$. Clearly,
$G$ is isomorphic to the graph $K_{n,n}$.

Now we define an edge-coloring $\alpha$ of $G$ as follows: for
$1\leq i\leq n$ and $1\leq j\leq n$, let
\begin{center}
$\alpha\left(u_{i}v_{j}\right)=\left\{
\begin{tabular}{ll}
$i+j-1\pmod{n}$, & if $i+j\neq n+1$,\\
$n$, & if $i+j=n+1$.
\end{tabular}%
\right.$
\end{center}

It is easy to see that $\alpha$ is a proper edge-coloring of $G$ and
$S(u_{i},\alpha)=S(v_{i},\alpha)=[1,n]$ for $1\leq i\leq n$.

Next we construct an interval total $(n\cdot l+1)$-coloring of
$K_{n,n\cdot l}$. Before we give the explicit definition of the
coloring, we need two auxiliary functions. For $i\in \mathbf{N}$, we
define a function $f_{1}(i)$ as follows: $f_{1}(i)=1+(i-1)\pmod{n}$.
For $j\in \mathbf{N}$, we define a function $f_{2}(j)$ as follows:
$f_{2}(j)=\left\lfloor\frac{j-1}{n}\right\rfloor$.

Now we able to define a total coloring $\beta$ of $K_{n,n\cdot l}$.

For $1\leq i\leq n$, let
\begin{center}
$\beta(u_{i})=n\cdot l+1$.
\end{center}

For $1\leq j\leq n\cdot l$, let
\begin{center}
$\beta(v_{j})=\left\{
\begin{tabular}{ll}
$n+1$, & if $1\leq j\leq n$,\\
$n\cdot f_{2}(j)$, & if $n+1\leq j\leq n\cdot l$.
\end{tabular}%
\right.$
\end{center}

For $1\leq i\leq n$ and $1\leq j\leq n\cdot l$, let
\begin{center}
$\beta(u_{i}v_{j})=\alpha\left(u_{f_{1}(i)}v_{f_{1}(j)}\right)+n\cdot
f_{2}(j)$.
\end{center}

Let us prove that $\beta$ is an interval total $(n\cdot
l+1)$-coloring of $K_{n,n\cdot l}$.

By the definition of $\beta$ and taking into account that
$S(u_{i},\alpha)=S(v_{i},\alpha)=[1,n]$ for $1\leq i\leq n$, we have

\begin{center}
$S[u_{i},\beta]=S(u_{i},\beta)\cup
\{\beta(u_{i})\}=\left(\bigcup_{k=1}^{l}S\left(u_{f_{1}(i)},\alpha\right)\oplus
n(k-1)\right)\cup \{\beta(u_{i})\}=$\\
$=[1,n\cdot l]\cup \{n\cdot l+1\}=[1,n\cdot l+1]$ for $1\leq i\leq n$,\\
$S[v_{j},\beta]=S(v_{j},\alpha)\cup
\{\beta(v_{j})\}=[1,n]\cup\{n+1\}=[1,n+1]$ for $1\leq j\leq n$, and\\
$S[v_{j},\beta]=S(v_{j},\beta)\cup \{\beta(v_{j})\}=
\left(S\left(v_{f_{1}(j)},\alpha\right)\oplus n\cdot
f_{2}(j)\right)\cup \{\beta(v_{j})\}=$\\
$=[1+n\cdot f_{2}(j),n+n\cdot f_{2}(j)]\cup \{n\cdot
f_{2}(j)\}=[n\cdot f_{2}(j),n+n\cdot f_{2}(j)]$ for $n+1\leq j\leq
n\cdot l$.
\end{center}

This shows that $\beta$ is an interval total $(n\cdot l+1)$-coloring
of $K_{n,n\cdot l}$. Thus, $w_{\tau }(K_{n,n\cdot l})\leq n\cdot
l+1$. On the other hand, by Theorem \ref{mytheorem2}, $w_{\tau
}(K_{n,n\cdot l})\geq \chi^{\prime\prime}(K_{n,n\cdot l})=n\cdot
l+1$ for $l\geq 2$ and hence $w_{\tau }(K_{n,n\cdot
l})=\chi^{\prime\prime}(K_{n,n\cdot l})$. $~\square$
\end{proof}

Next, we show that the difference between $w_{\tau }(K_{m,n})$ and
$\chi^{\prime\prime}(K_{m,n})$ for some $m$ and $n$ can be arbitrary
large.

\begin{theorem}
\label{mytheorem9} For any $l\in \mathbf{N}$, there exists a graph
$G$ such that $G\in {\cal T}$ and $w_{\tau
}(G)-\chi^{\prime\prime}(G)\geq l$.
\end{theorem}
\begin{proof} Let $n=l+3$. Clearly, $n\geq 4$. Consider the complete
bipartite graph $K_{n+l,2n}$ with bipartition $(X,Y)$, where $\vert
X\vert =n+l$ and $\vert Y\vert =2n$. By Theorem \ref{mytheorem2}, we
have $\chi^{\prime\prime}(K_{n+l,2n})=2n+1$.

By Theorem \ref{mytheorem4}, we have $K_{n+l,2n}\in {\cal T}$. We
now show that $w_{\tau }(K_{n+l,2n})\geq 2n+1+l$.

Suppose, to the contrary, that $K_{n+l,2n}$ has an interval total
$t$-coloring $\alpha$, where $2n+1\leq t\leq 2n+l$.

Let us consider $S[v,\alpha]$ for any $v\in V(K_{n+l,2n})$. It is
easy to see that $[t-n-l,n+l+1]\subseteq S[v,\alpha]$ for any $v\in
V(K_{n+l,2n})$. Since $t\leq 2n+l$, we have $[n,n+l+1]\subseteq
S[v,\alpha]$ for any $v\in V(K_{n+l,2n})$. This implies that for
each $c\in [n,n+l+1]$, there are $n-l$ vertices in $Y$ colored by
$c$. On the other hand, since $\vert Y\vert=2n$, we have $2n\geq
(l+2)(n-l)$, which contradicts the equality $n=l+3$.

This shows that $w_{\tau }(K_{n+l,2n})\geq 2n+1+l$. We take
$G=K_{n+l,2n}$. Hence, $w_{\tau }(G)-\chi^{\prime\prime}(G)\geq l$.
$~\square$
\end{proof}

Now we consider interval total colorings of complete $r$-partite
graphs with $n$ vertices in each part.

\begin{theorem}
\label{mytheorem10} If $r=2$ or $r$ is even and $n$ is odd, then
$K_{n,\ldots,n}\in {\cal T}$ and
\begin{center}
$w_{\tau }(K_{n,\ldots,n})\leq \left(\frac{3}{2}r-2\right)n+2$.
\end{center}
\end{theorem}
\begin{proof} First let us note that the theorem is true for the case $r=2$, since $w_{\tau
}(K_{n,n})=\chi^{\prime\prime}(K_{n,n})=n+2$ for any $n\in
\mathbf{N}$, by Theorem \ref{mytheorem5}.

Now we assume that $r$ is even and $n$ is odd. Clearly, for the
proof of the theorem, it suffices to prove that $K_{n,\ldots,n}$ has
an interval total
$\left(\left(\frac{3}{2}r-2\right)n+2\right)$-coloring.

Let $V(K_{n,\ldots,n})=\left\{v^{(i)}_{j}|~1\leq i\leq r,1\leq j\leq
n\right\}$ and

$E(K_{n,\ldots,n})=\left\{v^{(i)}_{p}v^{(j)}_{q}|~1\leq i<j\leq
r,1\leq p\leq n,1\leq q\leq n\right\}$.\\

Define a total coloring $\alpha$ of $K_{n,\ldots,n}$. First we color
the vertices of the graph as follows:

\begin{center}
$\alpha\left(v^{(1)}_{j}\right)=1$ for $1\leq j\leq n$ and
$\alpha\left(v^{(2)}_{j}\right)=(r-1)n+2$ for $1\leq j\leq n$,\\
$\alpha\left(v^{(i+1)}_{j}\right)=(i-1)n+1$ for $2\leq i\leq
\frac{r}{2}-1$ and $1\leq j\leq n$,\\
$\alpha\left(v^{(\frac{r}{2}+i-1)}_{j}\right)=(r+i-2)n+2$ for $2\leq
i\leq \frac{r}{2}-1$ and $1\leq j\leq n$,\\
$\alpha\left(v^{(r-1)}_{j}\right)=\left(\frac{r}{2}-1\right)n+1$ for
$1\leq j\leq n$ and
$\alpha\left(v^{(r)}_{j}\right)=\left(\frac{3}{2}r-2\right)n+2$ for
$1\leq j\leq n$.
\end{center}

Next we color the edges of the graph. For each edge
$v^{(i)}_{p}v^{(j)}_{q}\in E(K_{n,\ldots,n})$ with $1\leq i<j\leq r$
and $p=1,\ldots,n$, $q=1,\ldots,n$, define a
color $\alpha\left(v^{(i)}_{p}v^{(j)}_{q}\right)$ as follows:\\

for $i=1,\ldots,\left\lfloor\frac{r}{4}\right\rfloor$, $
j=2,\ldots,\frac{r}{2}$, $i+j\leq \frac{r}{2}+1$, let

\begin{center}
$\alpha \left(v_{p}^{(i)}v_{q}^{(j)}\right)=\left(
i+j-3\right)n+\left\{
\begin{tabular}{ll}
$1+(p+q-1)\pmod{n}$, & if $p+q\neq n+1$,\\
$n+1$, & if $p+q=n+1$;
\end{tabular}%
\right.$
\end{center}

for $i=2,\ldots,\frac{r}{2}-1$,
$j=\left\lfloor\frac{r}{4}\right\rfloor+2,\ldots,\frac{r}{2}$,
$i+j\geq \frac{r}{2}+2$, let

\begin{center}
$\alpha \left(v_{p}^{(i)}v_{q}^{(j)}\right)=\left(
i+j+\frac{r}{2}-4\right)n+\left\{
\begin{tabular}{ll}
$1+(p+q-1)\pmod{n}$, & if $p+q\neq n+1$,\\
$n+1$, & if $p+q=n+1$;
\end{tabular}%
\right.$
\end{center}

for $i=3,\ldots,\frac{r}{2}$, $j=\frac{r}{2}+1,\dots,r-2$, $j-i\leq
\frac{r}{2}-2$, let

\begin{center}
$\alpha \left(v_{p}^{(i)}v_{q}^{(j)}\right)=\left(
\frac{r}{2}+j-i-1\right)n+\left\{
\begin{tabular}{ll}
$1+(p+q-1)\pmod{n}$, & if $p+q\neq n+1$,\\
$n+1$, & if $p+q=n+1$;
\end{tabular}%
\right.$
\end{center}

for $i=1,\ldots,\frac{r}{2}$, $j=\frac{r}{2}+1,\ldots,r$, $j-i\geq
\frac{r}{2}$, let

\begin{center}
$\alpha
\left(v_{p}^{(i)}v_{q}^{(j)}\right)=\left(j-i-1\right)n+\left\{
\begin{tabular}{ll}
$1+(p+q-1)\pmod{n}$, & if $p+q\neq n+1$,\\
$n+1$, & if $p+q=n+1$;
\end{tabular}%
\right.$
\end{center}

for $i=2,\ldots,1+\left\lfloor \frac{r-2}{4}\right\rfloor$,
$j=\frac{r}{2}+1,\dots,\frac{r}{2}+\left\lfloor\frac{r-2}{4}\right\rfloor
$, $j-i=\frac{r}{2}-1$, let

\begin{center}
$\alpha
\left(v_{p}^{(i)}v_{q}^{(j)}\right)=\left(2i-3\right)n+\left\{
\begin{tabular}{ll}
$1+(p+q-1)\pmod{n}$, & if $p+q\neq n+1$,\\
$n+1$, & if $p+q=n+1$;
\end{tabular}%
\right.$
\end{center}

for $i=\left\lfloor
\frac{r-2}{4}\right\rfloor+2,\ldots,\frac{r}{2}$,
$j=\frac{r}{2}+1+\left\lfloor\frac{r-2}{4}\right\rfloor,\ldots,r-1$,
$j-i=\frac{r}{2}-1$, let

\begin{center}
$\alpha
\left(v_{p}^{(i)}v_{q}^{(j)}\right)=\left(i+j-3\right)n+\left\{
\begin{tabular}{ll}
$1+(p+q-1)\pmod{n}$, & if $p+q\neq n+1$,\\
$n+1$, & if $p+q=n+1$;
\end{tabular}%
\right.$
\end{center}

for $i=\frac{r}{2}+1,\dots,\frac{r}{2}+\left\lfloor
\frac{r}{4}\right\rfloor-1$, $j=\frac{r}{2}+2,\ldots,r-2$, $i+j\leq
\frac{3}{2}r-1$, let

\begin{center}
$\alpha
\left(v_{p}^{(i)}v_{q}^{(j)}\right)=\left(i+j-r-1\right)n+\left\{
\begin{tabular}{ll}
$1+(p+q-1)\pmod{n}$, & if $p+q\neq n+1$,\\
$n+1$, & if $p+q=n+1$;
\end{tabular}%
\right.$
\end{center}

for $i=\frac{r}{2}+1,\ldots,r-1$, $j=\frac{r}{2}+\left\lfloor
\frac{r}{4}\right\rfloor +1,\dots,r$, $i+j\geq \frac{3}{2}r$, let

\begin{center}
$\alpha \left(v_{p}^{(i)}v_{q}^{(j)}\right)=\left(
i+j-\frac{r}{2}-2\right)n+\left\{
\begin{tabular}{ll}
$1+(p+q-1)\pmod{n}$, & if $p+q\neq n+1$,\\
$n+1$, & if $p+q=n+1$.
\end{tabular}%
\right.$
\end{center}

Let us prove that $\alpha$ is an interval total
$\left(\left(\frac{3}{2}r-2\right)n+2\right)$-coloring of
$K_{n,\ldots,n}$.

First we show that for each $c\in
\left[1,\left(\frac{3}{2}r-2\right)n+2\right]$, there is $ve\in
VE(K_{n,\ldots,n})$ with $\alpha(ve)=c$.

Consider the vertices $v_{1}^{(1)}$ and $v_{1}^{(r)}$. By the
definition of $\alpha$, we have

\begin{center}
$S\left[v_{1}^{(1)},\alpha\right]
=S\left(v_{1}^{(1)},\alpha\right)\cup
\left\{\alpha\left(v_{1}^{(1)}\right)\right\}=\left(
\bigcup_{l=1}^{r-1}\left([2,n+1]\oplus (l-1)n\right)\right)\cup \{1\}=$\\
$=[2,(r-1)n+1]\cup \{1\}=[1,(r-1)n+1]$ and\\
$S\left[v_{1}^{(r)},\alpha\right]
=S\left(v_{1}^{(r)},\alpha\right)\cup
\left\{\alpha\left(v_{1}^{(r)}\right)\right\}=\left(
\bigcup_{l=\frac{r}{2}}^{\frac{3}{2}r-2}\left([2,n+1]\oplus
(l-1)n\right)\right)\cup \left\{\left(\frac{3}{2}r-2\right)n+2\right\}=$\\
$=\left[\left(\frac{r}{2}-1\right)n+2,\left(\frac{3}{2}r-2\right)n+1\right]\cup
\left\{\left(\frac{3}{2}r-2\right)n+2\right\}=\left[\left(\frac{r}{2}-1\right)n+2,\left(\frac{3}{2}r-2\right)n+2\right]$.
\end{center}

It is straightforward to check that
$S\left[v_{1}^{(1)},\alpha\right]\cup
S\left[v_{1}^{(r)},\alpha\right]=\left[1,\left(\frac{3}{2}r-2\right)n+2\right]$,
so for each $c\in \left[1,\left(\frac{3}{2}r-2\right)n+2\right]$,
there is $ve\in VE(K_{n,\ldots,n})$ with $\alpha(ve)=c$.

Next we show that the edges incident to each vertex of
$K_{n,\ldots,n}$ together with this vertex are colored by $(r-1)n+1$
consecutive colors.

Let $v_{j}^{(i)}\in V(K_{n,\ldots,n})$, where $1\leq i\leq r$,
$1\leq j\leq n$.

Case 1. $1\leq i\leq 2$, $1\leq j\leq n$.

By the definition of $\alpha$, we have

\begin{center}
$S\left[v_{j}^{(1)},\alpha\right]=S\left(v_{j}^{(1)},\alpha\right)\cup
\left\{\alpha\left(v_{j}^{(1)}\right)\right\}=\left(
\bigcup_{l=1}^{r-1}\left([2,n+1]\oplus (l-1)n\right)\right)\cup \{1\}=$\\
$=[2,(r-1)n+1]\cup \{1\}=[1,(r-1)n+1]$ and\\
$S\left[v_{j}^{(2)},\alpha\right]=S\left(v_{j}^{(2)},\alpha\right)\cup
\left\{\alpha\left(v_{j}^{(2)}\right)\right\}=\left(
\bigcup_{l=1}^{r-1}\left([2,n+1]\oplus (l-1)n\right)\right)\cup \{(r-1)n+2\}=$\\
$=[2,(r-1)n+1]\cup \{(r-1)n+2\}=[2,(r-1)n+2]$.
\end{center}

Case 2. $3\leq i\leq \frac{r}{2}$, $1\leq j\leq n$.

By the definition of $\alpha$, we have

\begin{center}
$S\left[v_{j}^{(i)},\alpha\right]=S\left(v_{j}^{(i)},\alpha\right)\cup
\left\{\alpha\left(v_{j}^{(i)}\right)\right\}
=\left(\bigcup_{l=i-1}^{r-3+i}\left([2,n+1]\oplus
(l-1)n\right)\right)\cup
\{(i-2)n+1\}=$\\
$=\left[(i-2)n+2,(r-3+i)n+1\right]\cup
\{(i-2)n+1\}=[(i-2)n+1,(r-3+i)n+1]$.
\end{center}

Case 3. $\frac{r}{2}+1\leq i\leq r-2$, $1\leq j\leq n$.

By the definition of $\alpha$, we have

\begin{center}
$S\left[v_{j}^{(i)},\alpha\right]=S\left(v_{j}^{(i)},\alpha\right)\cup
\left\{\alpha\left(v_{j}^{(i)}\right)\right\}
=\left(\bigcup_{l=i-\frac{r}{2}+1}^{\frac{r}{2}-1+i}\left([2,n+1]\oplus
(l-1)n\right)\right)\cup
\left\{\left(\frac{r}{2}+i-1\right)n+2\right\}=
\left[\left(i-\frac{r}{2}\right)n+2,\left(\frac{r}{2}+i-1\right)n+1\right]\cup
\left\{\left(\frac{r}{2}+i-1\right)n+2\right\}=\left[\left(i-\frac{r}{2}\right)n+2,\left(\frac{r}{2}+i-1\right)n+2\right]$.
\end{center}

Case 4. $r-1\leq i\leq r,1\leq j\leq n$.

By the definition of $\alpha$, we have

\begin{center}
$S\left[v_{j}^{(r-1)},\alpha\right]=S\left(v_{j}^{(r-1)},\alpha\right)\cup
\left\{\alpha\left(v_{j}^{(r-1)}\right)\right\}=\left(\bigcup_{l=\frac{r}{2}}^{\frac{3}{2}r-2}\left([2,n+1]\oplus
(l-1)n\right)\right)\cup
\left\{\left(\frac{r}{2}-1\right)n+1\right\}=\left[\left(\frac{r}{2}-1\right)n+2,\left(\frac{3}{2}r-2\right)n+1\right]\cup
\left\{\left(\frac{r}{2}-1\right)n+1\right\}=\left[\left(\frac{r}{2}-1\right)n+1,\left(\frac{3}{2}r-2\right)n+1\right]$
and\\
$S\left[v_{j}^{(r)},\alpha\right]=S\left(v_{j}^{(r)},\alpha\right)\cup
\left\{\alpha\left(v_{j}^{(r)}\right)\right\}=\left(\bigcup_{l=\frac{r}{2}}^{\frac{3}{2}r-2}\left([2,n+1]\oplus
(l-1)n\right)\right)\cup
\left\{\left(\frac{3}{2}r-2\right)n+2\right\}=\left[\left(\frac{r}{2}-1\right)n+2,\left(\frac{3}{2}r-2\right)n+1\right]\cup
\left\{\left(\frac{3}{2}r-2\right)n+2\right\}=\left[\left(\frac{r}{2}-1\right)n+2,\left(\frac{3}{2}r-2\right)n+2\right]$.
\end{center}

This shows that $\alpha$ is an interval total
$\left(\left(\frac{3}{2}r-2\right)n+2\right)$-coloring of
$K_{n,\ldots,n}$; thus $K_{n,\ldots,n}\in {\cal T}$ and $w_{\tau
}(K_{n,\ldots,n})\leq \left(\frac{3}{2}r-2\right)n+2$. $~\square$
\end{proof}

Note that the upper bound in Theorem \ref{mytheorem10} is sharp when
$r=2$ or $n=1$. Also, by Theorem \ref{mytheorem10}, we have that if
$r=2$ or $r$ is even and $n$ is odd, then $K_{n,\ldots,n}\in {\cal
T}$; otherwise, by Theorem \ref{mytheorem3} and taking into account
that $K_{n,\ldots,n}$ is an $(r-1)n$-regular graph, we have
$K_{n,\ldots,n}\in {\cal T}$ and $w_{\tau
}(K_{n,\ldots,n})=\chi^{\prime\prime}(K_{n,\ldots,n})=(r-1)n+1$. Our
next result gives a lower bound for $W_{\tau }(K_{n,\ldots,n})$ when
$n\cdot r$ is even.

\begin{theorem}
\label{mytheorem11} If $r\geq 2,n\in \mathbf{N}$ and $n\cdot r$ is
even, then $W_{\tau }(K_{n,\ldots,n})\geq
\left(\frac{3}{2}r-1\right)n+1$.
\end{theorem}
\begin{proof} We distinguish our proof into two cases.

Case 1: $r$ is even.

Let $V(K_{n,\ldots,n})=\left\{v^{(i)}_{j}|~1\leq i\leq r,1\leq j\leq
n\right\}$ and

$E(K_{n,\ldots,n})=\left\{v^{(i)}_{p}v^{(j)}_{q}|~1\leq i<j\leq
r,1\leq p\leq n,1\leq q\leq n\right\}$.\\

Define a total coloring $\alpha$ of $K_{n,\ldots,n}$. First we color
the vertices of the graph as follows:

\begin{center}
$\alpha\left(v^{(1)}_{j}\right)=j$ for $1\leq j\leq n$ and
$\alpha\left(v^{(2)}_{j}\right)=(r-1)n+1+j$ for $1\leq j\leq n$,\\
$\alpha\left(v^{(i+1)}_{j}\right)=(i-1)n+j$ for $2\leq i\leq
\frac{r}{2}-1$ and $1\leq j\leq n$,\\
$\alpha\left(v^{(\frac{r}{2}+i-1)}_{j}\right)=(r+i-2)n+1+j$ for
$2\leq
i\leq \frac{r}{2}-1$ and $1\leq j\leq n$,\\
$\alpha\left(v^{(r-1)}_{j}\right)=\left(\frac{r}{2}-1\right)n+j$ for
$1\leq j\leq n$ and
$\alpha\left(v^{(r)}_{j}\right)=\left(\frac{3}{2}r-2\right)n+1+j$
for $1\leq j\leq n$.
\end{center}

Next we color the edges of the graph. For each edge
$v^{(i)}_{p}v^{(j)}_{q}\in E(K_{n,\ldots,n})$ with $1\leq i<j\leq r$
and $p=1,\ldots,n$, $q=1,\ldots,n$, define a
color $\alpha\left(v^{(i)}_{p}v^{(j)}_{q}\right)$ as follows:\\

for $i=1,\ldots,\left\lfloor\frac{r}{4}\right\rfloor$, $
j=2,\ldots,\frac{r}{2}$, $i+j\leq \frac{r}{2}+1$, let

\begin{center}
$\alpha \left(v_{p}^{(i)}v_{q}^{(j)}\right)=\left(
i+j-3\right)n+p+q$;
\end{center}

for $i=2,\ldots,\frac{r}{2}-1$,
$j=\left\lfloor\frac{r}{4}\right\rfloor+2,\ldots,\frac{r}{2}$,
$i+j\geq \frac{r}{2}+2$, let

\begin{center}
$\alpha \left(v_{p}^{(i)}v_{q}^{(j)}\right)=\left(
i+j+\frac{r}{2}-4\right)n+p+q$;
\end{center}

for $i=3,\ldots,\frac{r}{2}$, $j=\frac{r}{2}+1,\dots,r-2$, $j-i\leq
\frac{r}{2}-2$, let

\begin{center}
$\alpha \left(v_{p}^{(i)}v_{q}^{(j)}\right)=\left(
\frac{r}{2}+j-i-1\right)n+p+q$;
\end{center}

for $i=1,\ldots,\frac{r}{2}$, $j=\frac{r}{2}+1,\ldots,r$, $j-i\geq
\frac{r}{2}$, let

\begin{center}
$\alpha
\left(v_{p}^{(i)}v_{q}^{(j)}\right)=\left(j-i-1\right)n+p+q$;
\end{center}

for $i=2,\ldots,1+\left\lfloor \frac{r-2}{4}\right\rfloor$,
$j=\frac{r}{2}+1,\dots,\frac{r}{2}+\left\lfloor\frac{r-2}{4}\right\rfloor
$, $j-i=\frac{r}{2}-1$, let

\begin{center}
$\alpha \left(v_{p}^{(i)}v_{q}^{(j)}\right)=\left(2i-3\right)n+p+q$;
\end{center}

for $i=\left\lfloor
\frac{r-2}{4}\right\rfloor+2,\ldots,\frac{r}{2}$,
$j=\frac{r}{2}+1+\left\lfloor\frac{r-2}{4}\right\rfloor,\ldots,r-1$,
$j-i=\frac{r}{2}-1$, let

\begin{center}
$\alpha
\left(v_{p}^{(i)}v_{q}^{(j)}\right)=\left(i+j-3\right)n+p+q$;
\end{center}

for $i=\frac{r}{2}+1,\dots,\frac{r}{2}+\left\lfloor
\frac{r}{4}\right\rfloor-1$, $j=\frac{r}{2}+2,\ldots,r-2$, $i+j\leq
\frac{3}{2}r-1$, let

\begin{center}
$\alpha
\left(v_{p}^{(i)}v_{q}^{(j)}\right)=\left(i+j-r-1\right)n+p+q$;
\end{center}

for $i=\frac{r}{2}+1,\ldots,r-1$, $j=\frac{r}{2}+\left\lfloor
\frac{r}{4}\right\rfloor +1,\dots,r$, $i+j\geq \frac{3}{2}r$, let

\begin{center}
$\alpha \left(v_{p}^{(i)}v_{q}^{(j)}\right)=\left(
i+j-\frac{r}{2}-2\right)n+p+q$.
\end{center}

Let us prove that $\alpha$ is an interval total
$\left(\left(\frac{3}{2}r-1\right)n+1\right)$-coloring of
$K_{n,\ldots,n}$.

First we show that for each $c\in
\left[1,\left(\frac{3}{2}r-1\right)n+1\right]$, there is $ve\in
VE(K_{n,\ldots,n})$ with $\alpha(ve)=c$.

Consider the vertices $v_{1}^{(1)},\ldots,v_{n}^{(1)}$ and
$v_{1}^{(r)},\ldots,v_{n}^{(r)}$. By the definition of $\alpha$, for
$1\leq j\leq n$, we have

\begin{center}
$S\left[v_{j}^{(1)},\alpha\right]=S\left(v_{j}^{(1)},\alpha\right)\cup
\left\{\alpha\left(v_{j}^{(1)}\right)\right\}
=\left(\bigcup_{l=1}^{r-1}\left([j+1,j+n]\oplus
(l-1)n\right)\right)\cup \{j\}=$\\
$=[j+1,(r-1)n+j]\cup \{j\}=[j,(r-1)n+j]$ and
$S\left[v_{j}^{(r)},\alpha\right]=S\left(v_{j}^{(r)},\alpha\right)\cup
\left\{\alpha\left(v_{j}^{(r)}\right)\right\}=$\\
$=\left(\bigcup_{l=\frac{r}{2}}^{\frac{3}{2}r-2}\left([j+1,j+n]\oplus
(l-1)n\right)\right)\cup
\left\{\left(\frac{3}{2}r-2\right)n+1+j\right\}=$\\
$=\left[\left(\frac{r}{2}-1\right)n+1+j,\left(\frac{3}{2}r-2\right)n+j\right]\cup
\left\{\left(\frac{3}{2}r-2\right)n+1+j\right\}=$\\
$=\left[\left(\frac{r}{2}-1\right)n+1+j,\left(\frac{3}{2}r-2\right)n+1+j\right]$.
\end{center}

Let $\overline{C}=\bigcup_{j=1}^{n}
S\left[v_{j}^{(1)},\alpha\right]$ and
$\overline{\overline{C}}=\bigcup_{j=1}^{n}
S\left[v_{j}^{(r)},\alpha\right]$. It is straightforward to check
that $\overline{C}\cup
\overline{\overline{C}}=\left[1,\left(\frac{3}{2}r-1\right)n+1\right]$,
so for each $c\in \left[1,\left(\frac{3}{2}k-1\right)n+1\right]$,
there is $ve\in VE(K_{n,\ldots,n})$ with $\alpha(ve)=c$.

Next we show that the edges incident to each vertex of
$K_{n,\ldots,n}$ together with this vertex are colored by $(r-1)n+1$
consecutive colors.

Let $v_{j}^{(i)}\in V(K_{n,\ldots,n})$, where $1\leq i\leq r$,
$1\leq j\leq n$.

Subcase 1.1. $1\leq i\leq 2$, $1\leq j\leq n$.

By the definition of $\alpha$, we have

\begin{center}
$S\left[v_{j}^{(1)},\alpha\right]=S\left(v_{j}^{(1)},\alpha\right)\cup
\left\{\alpha\left(v_{j}^{(1)}\right)\right\}=\left(
\bigcup_{l=1}^{r-1}\left([j+1,j+n]\oplus (l-1)n\right)\right)\cup \{j\}=$\\
$=[j+1,(r-1)n+j]\cup \{j\}=[j,(r-1)n+j]$ and\\
$S\left[v_{j}^{(2)},\alpha\right]=S\left(v_{j}^{(2)},\alpha\right)\cup
\left\{\alpha\left(v_{j}^{(2)}\right)\right\}=\left(
\bigcup_{l=1}^{r-1}\left([j+1,j+n]\oplus (l-1)n\right)\right)\cup \{(r-1)n+1+j\}=$\\
$=[j+1,(r-1)n+1+j]$.
\end{center}

Subcase 1.2. $3\leq i\leq \frac{r}{2}$, $1\leq j\leq n$.

By the definition of $\alpha$, we have

\begin{center}
$S\left[v_{j}^{(i)},\alpha\right]=S\left(v_{j}^{(i)},\alpha\right)\cup
\left\{\alpha\left(v_{j}^{(i)}\right)\right\}
=\left(\bigcup_{l=i-1}^{r-3+i}\left([j+1,j+n]\oplus
(l-1)n\right)\right)\cup
\{(i-2)n+j\}=$\\
$=\left[(i-2)n+1+j,(r-3+i)n+j\right]\cup
\{(i-2)n+j\}=[(i-2)n+j,(r-3+i)n+j]$.
\end{center}

Subcase 1.3. $\frac{r}{2}+1\leq i\leq r-2$, $1\leq j\leq n$.

By the definition of $\alpha$, we have

\begin{center}
$S\left[v_{j}^{(i)},\alpha\right]=S\left(v_{j}^{(i)},\alpha\right)\cup
\left\{\alpha\left(v_{j}^{(i)}\right)\right\}=$\\
$=\left(\bigcup_{l=i-\frac{r}{2}+1}^{\frac{r}{2}-1+i}\left([j+1,j+n]\oplus
(l-1)n\right)\right)\cup
\left\{\left(\frac{r}{2}+i-1\right)n+1+j\right\}=$\\
$=\left[\left(i-\frac{r}{2}\right)n+1+j,\left(\frac{r}{2}+i-1\right)n+j\right]\cup
\left\{\left(\frac{r}{2}+i-1\right)n+1+j\right\}=\left[\left(i-\frac{r}{2}\right)n+1+j,\left(\frac{r}{2}+i-1\right)n+1+j\right]$.
\end{center}

Subcase 1.4. $r-1\leq i\leq r,1\leq j\leq n$.

By the definition of $\alpha$, we have

\begin{center}
$S\left[v_{j}^{(r-1)},\alpha\right]=S\left(v_{j}^{(r-1)},\alpha\right)\cup
\left\{\alpha\left(v_{j}^{(r-1)}\right)\right\}=\left(\bigcup_{l=\frac{r}{2}}^{\frac{3}{2}r-2}\left([j+1,j+n]\oplus
(l-1)n\right)\right)\cup
\left\{\left(\frac{r}{2}-1\right)n+j\right\}=\left[\left(\frac{r}{2}-1\right)n+1+j,\left(\frac{3}{2}r-2\right)n+j\right]\cup
\left\{\left(\frac{r}{2}-1\right)n+j\right\}=\left[\left(\frac{r}{2}-1\right)n+j,\left(\frac{3}{2}r-2\right)n+j\right]$
and\\
$S\left[v_{j}^{(r)},\alpha\right]=S\left(v_{j}^{(r)},\alpha\right)\cup
\left\{\alpha\left(v_{j}^{(r)}\right)\right\}=\left(\bigcup_{l=\frac{r}{2}}^{\frac{3}{2}r-2}\left([j+1,j+n]\oplus
(l-1)n\right)\right)\cup
\left\{\left(\frac{3}{2}r-2\right)n+1+j\right\}=\left[\left(\frac{r}{2}-1\right)n+1+j,\left(\frac{3}{2}r-2\right)n+j\right]\cup
\left\{\left(\frac{3}{2}r-2\right)n+1+j\right\}=\left[\left(\frac{r}{2}-1\right)n+1+j,\left(\frac{3}{2}r-2\right)n+1+j\right]$.
\end{center}

This shows that $\alpha$ is an interval total
$\left(\left(\frac{3}{2}r-1\right)n+1\right)$-coloring of
$K_{n,\ldots,n}$; thus $W_{\tau }(K_{n,\ldots,n})\geq
\left(\frac{3}{2}r-1\right)n+1$ for even $r\geq 2$.

Case 2: $n$ is even.

Let $n=2m$, $m\in \mathbf{N}$. Let
$S_{i}=\left\{u_{1}^{(i)},\ldots,u_{m}^{(i)},u_{1}^{(r+i)},\ldots,u_{m}^{(r+i)}\right\}$
($1\leq i\leq r$) be the $r$ independent sets of vertices of
$K_{n,\ldots,n}$. For $i=1,\ldots,2r$, define the set $U_{i}$ as
follows: $U_{i}=\left\{u_{1}^{(i)},\ldots,u_{m}^{(i)}\right\}$.
Clearly, $V(K_{n,\ldots,n})=\bigcup_{i=1}^{2r}U_{i}$. For $1\leq
i<j\leq 2r$, define $(U_{i},U_{j})$ as the set of all edges between
$U_{i}$ and $U_{j}$. It is easy to see that for $1\leq i<j\leq 2r$,
$\left\vert (U_{i},U_{j})\right\vert=m^{2}$ except for $\left\vert
(U_{i},U_{r+i})\right\vert=0$ whenever $i=1,\ldots,r$. If we
consider the sets $U_{i}$ as the vertices and the sets
$(U_{i},U_{j})$ as the edges, then we obtain that $K_{n,\ldots,n}$
is isomorphic to the graph $K_{2r}-F$, where $F$ is a perfect
matching of $K_{2r}$. Now we define a total coloring $\beta$ of the
graph $K_{n,\ldots,n}$. First we color the vertices of the graph as
follows:

\begin{center}
$\beta\left(u^{(1)}_{j}\right)=j$ for $1\leq j\leq m$ and
$\beta\left(u^{(2)}_{j}\right)=(2r-2)m+1+j$ for $1\leq j\leq m$,\\
$\beta\left(u^{(i+1)}_{j}\right)=(i-1)m+j$ for $2\leq i\leq
r-1$ and $1\leq j\leq m$,\\
$\beta\left(u^{(r+i-1)}_{j}\right)=(2r+i-3)m+1+j$ for $2\leq
i\leq r-1$ and $1\leq j\leq m$,\\
$\beta\left(u^{(2r-1)}_{j}\right)=(r-1)m+j$ for $1\leq j\leq m$ and
$\beta\left(u^{(2r)}_{j}\right)=(3r-3)m+1+j$ for $1\leq j\leq m$.
\end{center}

Next we color the edges of the graph. For each edge
$u^{(i)}_{p}u^{(j)}_{q}\in E(K_{n,\ldots,n})$ with $1\leq i<j\leq
2r$ and $p=1,\ldots,m$, $q=1,\ldots,m$, define a
color $\beta\left(u^{(i)}_{p}u^{(j)}_{q}\right)$ as follows:\\

for $i=1,\ldots,\left\lfloor\frac{r}{2}\right\rfloor$, $
j=2,\ldots,r$, $i+j\leq r+1$, let

\begin{center}
$\beta \left(u_{p}^{(i)}u_{q}^{(j)}\right)=\left(
i+j-3\right)m+p+q$;\\
\end{center}

for $i=2,\ldots,r-1$,
$j=\left\lfloor\frac{r}{2}\right\rfloor+2,\ldots,r$, $i+j\geq r+2$,
let

\begin{center}
$\beta\left(u_{p}^{(i)}u_{q}^{(j)}\right)=\left(
i+j+r-5\right)m+p+q$;\\
\end{center}

for $i=3,\ldots,r$, $j=r+1,\dots,2r-2$, $j-i\leq r-2$, let

\begin{center}
$\beta\left(u_{p}^{(i)}u_{q}^{(j)}\right)=\left(
r+j-i-2\right)m+p+q$;\\
\end{center}

for $i=1,\ldots,r-1$, $j=r+2,\ldots,2r$, $j-i\geq r+1$, let

\begin{center}
$\beta
\left(u_{p}^{(i)}u_{q}^{(j)}\right)=\left(j-i-2\right)m+p+q$;\\
\end{center}

for $i=2,\ldots,1+\left\lfloor \frac{r-1}{2}\right\rfloor$,
$j=r+1,\dots,r+\left\lfloor\frac{r-1}{2}\right\rfloor $, $j-i=r-1$,
let

\begin{center}
$\beta
\left(u_{p}^{(i)}u_{q}^{(j)}\right)=\left(2i-3\right)m+p+q$;\\
\end{center}

for $i=\left\lfloor \frac{r-1}{2}\right\rfloor+2,\ldots,r$,
$j=r+1+\left\lfloor\frac{r-1}{2}\right\rfloor,\ldots,2r-1$,
$j-i=r-1$, let

\begin{center}
$\beta
\left(u_{p}^{(i)}u_{q}^{(j)}\right)=\left(i+j-4\right)m+p+q$;\\
\end{center}

for $i=r+1,\dots,r+\left\lfloor\frac{r}{2}\right\rfloor-1$,
$j=r+2,\ldots,2r-2$, $i+j\leq 3r-1$, let

\begin{center}
$\beta
\left(u_{p}^{(i)}u_{q}^{(j)}\right)=\left(i+j-2r-1\right)m+p+q$;\\
\end{center}

for $i=r+1,\ldots,2r-1$, $j=r+\left\lfloor \frac{r}{2}\right\rfloor
+1,\dots,2r$, $i+j\geq 3r$, let

\begin{center}
$\beta \left(u_{p}^{(i)}u_{q}^{(j)}\right)=\left(
i+j-r-3\right)m+p+q$.\\
\end{center}

Let us prove that $\beta$ is an interval total
$\left(\left(\frac{3}{2}r-1\right)n+1\right)$-coloring of the graph
$K_{n,\ldots,n}$.

First we show that for each $c\in
\left[1,\left(\frac{3}{2}r-1\right)n+1\right]$, there is $ve\in
VE(K_{n,\ldots,n})$ with $\beta(ve)=c$.

Consider the vertices $u_{1}^{(1)},\ldots,u_{m}^{(1)}$ and
$u_{1}^{(2r)},\ldots,u_{m}^{(2r)}$. By the definition of $\beta$,
for $1\leq j\leq m$, we have

\begin{center}
$S\left[u_{j}^{(1)},\beta\right]=S\left(u_{j}^{(1)},\beta\right)\cup
\left\{\beta\left(u^{(1)}_{j}\right)\right\}
=\left(\bigcup_{l=1}^{2r-2}\left([j+1,j+m]\oplus
(l-1)m\right)\right)\cup\{j\}=[j+1,(2r-2)m+j]\cup\{j\}=[j,(2r-2)m+j]$
and
$S\left[u_{j}^{(2r)},\beta\right]=S\left(u_{j}^{(2r)},\beta\right)\cup
\left\{\beta\left(u^{(2r)}_{j}\right)\right\}
=\left(\bigcup_{l=r}^{3r-3}\left([j+1,j+m]\oplus
(l-1)m\right)\right)\cup
\{\left(3r-3\right)m+1+j\}=[(r-1)m+1+j,(3r-3)m+j]\cup\{(3r-3)m+1+j\}=[(r-1)m+1+j,(3r-3)m+1+j]$.
\end{center}

Let $\tilde{C}=\bigcup_{j=1}^{m} S\left[u_{j}^{(1)},\beta\right]$
and $\tilde{\tilde{C}}=\bigcup_{j=1}^{m}
S\left[u_{j}^{(2r)},\beta\right]$. It is straightforward to check
that $\tilde{C}\cup
\tilde{\tilde{C}}=\left[1,\left(\frac{3}{2}r-1\right)n+1\right]$, so
for each $c\in \left[1,\left(\frac{3}{2}r-1\right)n+1\right]$, there
is $ve\in VE(K_{n,\ldots,n})$ with $\beta(ve)=c$.

Next we show that the edges incident to each vertex of
$K_{n,\ldots,n}$ together with this vertex are colored by $(r-1)n+1$
consecutive colors.

Let $v_{j}^{(i)}\in V(K_{n,\ldots,n})$, where $1\leq i\leq 2r$,
$1\leq j\leq m$.

Subcase 2.1. $1\leq i\leq 2$, $1\leq j\leq m$.

By the definition of $\beta$, we have

\begin{center}
$S\left[v_{j}^{(1)},\beta\right]=S\left(u_{j}^{(1)},\beta\right)\cup
\left\{\beta\left(u^{(1)}_{j}\right)\right\}
=\left(\bigcup_{l=1}^{2r-2}\left([j+1,j+m]\oplus
(l-1)m\right)\right)\cup\{j\}=[j+1,(2r-2)m+j]\cup\{j\}=[j,(2r-2)m+j]$
and
$S\left[u_{j}^{(2)},\beta\right]=S\left(u_{j}^{(2)},\beta\right)\cup
\left\{\beta\left(u^{(2)}_{j}\right)\right\}
=\left(\bigcup_{l=1}^{2r-2}\left([j+1,j+m]\oplus
(l-1)m\right)\right)\cup\{(2r-2)m+1+j\}=[j+1,(2r-2)m+j]\cup\{(2r-2)m+1+j\}=[j+1,(2r-2)m+1+j]$.
\end{center}

Subcase 2.2. $3\leq i\leq r$, $1\leq j\leq m$.

By the definition of $\beta$, we have

\begin{center}
$S\left[u_{j}^{(i)},\beta\right]=S\left(u_{j}^{(i)},\beta\right)\cup
\left\{\beta\left(u^{(i)}_{j}\right)\right\}
=\left(\bigcup_{l=i-1}^{2r-4+i}\left([j+1,j+m]\oplus
(l-1)m\right)\right)\cup \{(i-2)m+j\}
=\left[(i-2)m+1+j,(2r-4+i)m+j\right]\cup \{(i-2)m+j\}
=\left[(i-2)m+j,(2r-4+i)m+j\right]$.
\end{center}

Subcase 2.3. $r+1\leq i\leq 2r-2$, $1\leq j\leq m$.

By the definition of $\beta$, we have

\begin{center}
$S\left[u_{j}^{(i)},\beta\right]=S\left(u_{j}^{(i)},\beta\right)\cup
\left\{\beta\left(u^{(i)}_{j}\right)\right\}
=\left(\bigcup_{l=i-r+1}^{r-2+i}\left([j+1,j+m]\oplus
(l-1)m\right)\right)\cup
\{(r+i-2)m+1+j\}=\left[(i-r)m+1+j,(r+i-2)m+j\right]\cup
\{(r+i-2)m+1+j\}=\left[(i-r)m+1+j,(r+i-2)m+1+j\right]$.
\end{center}

Subcase 2.4. $2r-1\leq i\leq 2r,1\leq j\leq m$.

By the definition of $\beta$, we have
\begin{center}
$S\left[u_{j}^{(2r-1)},\beta\right]=S\left(u_{j}^{(2r-1)},\beta\right)\cup
\left\{\beta\left(u^{(2r-1)}_{j}\right)\right\}
=\left(\bigcup_{l=r}^{3r-3}\left([j+1,j+m]\oplus
(l-1)m\right)\right)\cup
\{(r-1)m+j\}=[(r-1)m+1+j,(3r-3)m+j]\cup\{(r-1)m+j\}=[(r-1)m+j,(3r-3)m+j]$
and
$S\left[u_{j}^{(2r)},\beta\right]=S\left(u_{j}^{(2r)},\beta\right)\cup
\left\{\beta\left(u^{(2r)}_{j}\right)\right\}
=\left(\bigcup_{l=r}^{3r-3}\left([j+1,j+m]\oplus
(l-1)m\right)\right)\cup
\{\left(3r-3\right)m+1+j\}=[(r-1)m+1+j,(3r-3)m+j]\cup\{(3r-3)m+1+j\}=[(r-1)m+1+j,(3r-3)m+1+j]$.
\end{center}

This shows that $\beta$ is an interval total
$\left(\left(\frac{3}{2}r-1\right)n+1\right)$-coloring of
$K_{n,\ldots,n}$; thus $W_{\tau }(K_{n,\ldots,n})\geq
\left(\frac{3}{2}r-1\right)n+1$ for even $n\geq 2$. $~\square$
\end{proof}
\bigskip

\section{Interval total colorings of hypercubes}\

In \cite{Petros3}, the first author investigated interval colorings
of hypercubes $Q_{n}$. In particular, he proved that $Q_{n}\in {\cal
N}$ and $w(Q_{n})=n$, $W(Q_{n})\geq\frac{n(n+1)}{2}$ for any
$n\in\mathbf{N}$. In the same paper he also conjectured that
$W(Q_{n}) =\frac{n(n+1)}{2}$ for any $n\in\mathbf{N}$. In
\cite{PetKhachTan}, the authors confirmed this conjecture. Here, we
prove that $W_{\tau}(Q_{n})=\frac{(n+1)(n+2)}{2}$ for any
$n\in\mathbf{N}$.

\begin{theorem}
\label{mytheorem12} If $n\in\mathbf{N}$, then
$W_{\tau}(Q_{n})=\frac{(n+1)(n+2)}{2}$.
\end{theorem}
\begin{proof} First of all let us note that $W_{\tau}(Q_{n})\geq \frac{(n+1)(n+2)}{2}$ for any $n\in\mathbf{N}$, by Theorem \ref{mytheorem7}.
For the proof of the theorem, it suffices to show that
$W_{\tau}(Q_{n})\leq \frac{(n+1)(n+2)}{2}$ for any $n\in\mathbf{N}$.
Let $\varphi$ be an interval total $W_{\tau}(Q_{n})$-coloring of
$Q_{n}$.

Let $i=0$ or $1$ and $Q_{n+1}^{(i)}$ be a subgraph of the graph
$Q_{n+1}$, induced by the vertices $\left\{\left( i,\alpha
_{2},\alpha _{3},\ldots ,\alpha _{n+1}\right)|\left( \alpha
_{2},\alpha _{3},\ldots ,\alpha _{n+1}\right) \in \left\{
0,1\right\}^{n}\right\}$. Clearly, $Q_{n+1}^{(i)}$ is isomorphic to
$Q_{n}$ for $i\in \{0,1\}$.\\

Let us define an edge coloring $\psi $ of the graph $Q_{n+1}$ in the
following way:

\begin{description}
\item[(1)] for $i=0,1$ and every edge $\left(i,\bar
\alpha\right)\left(i,\bar\beta\right)\in
E\left(Q_{n+1}^{(i)}\right)$, let
\begin{center}
$\psi \left(\left(i,\bar
\alpha\right)\left(i,\bar\beta\right)\right)=\varphi \left(\bar
\alpha \bar\beta\right)$;
\end{center}

\item[(2)] for every $\bar
\alpha\in \left\{0,1\right\}^{n}$, let
\begin{center}
$\psi \left(\left(0,\bar
\alpha\right)\left(1,\bar\alpha\right)\right)=\varphi \left(\bar
\alpha \right)$.
\end{center}
\end{description}

It is not difficult to see that $\psi$ is an interval
$W_{\tau}(Q_{n})$-coloring of the graph $Q_{n+1}$. Thus,
$W_{\tau}(Q_{n})\leq W(Q_{n+1})=\frac{(n+1)(n+2)}{2}$ for any
$n\in\mathbf{N}$. $~\square$
\end{proof}

By Theorems \ref{mytheorem7} and \ref{mytheorem12}, we have that
$Q_{n}$ has an interval total $t$-coloring if and only if
$w_{\tau}(Q_{n})\leq t\leq W_{\tau}(Q_{n})$.

\end{document}